\documentclass[a4paper,11pt]{article}

\newcommand{\Mod}[1]{\ (\mathrm{mod}\ #1)}

\usepackage{epsf,epsfig,amsfonts,amsgen,amsmath,amstext,amsbsy,amsopn,amsthm,cases,listings,color
}
\usepackage{ebezier,eepic}
\usepackage{color}
\usepackage{multirow}
\usepackage{epstopdf}
\usepackage{graphicx}
\usepackage{pgf,tikz}
\usepackage{mathrsfs}
\usepackage[marginal]{footmisc}
\usepackage{enumitem}
\usepackage[titletoc]{appendix}
\usepackage{booktabs}
\usepackage{url}
\usepackage{mathtools}

\usepackage{pgfplots}
\usepackage{authblk}
\usepackage{amssymb}

\usepackage{wasysym}

\usepackage{empheq}

\usepackage{dsfont}

\usepackage{subfigure}
\pgfplotsset{compat=1.18}
\usepackage{mathrsfs}
\usepackage{wasysym} 
\usetikzlibrary{arrows}
\usepackage{aligned-overset}
\usepackage{bm}
\usepackage[backref=page]{hyperref}

\allowdisplaybreaks[1]

\definecolor{uuuuuu}{rgb}{0.27,0.27,0.27}
\definecolor{sqsqsq}{rgb}{0.1255,0.1255,0.1255}

\newtheorem{definition}{Definition} [section]
\newtheorem{theorem}[definition]{Theorem}

\newtheorem{claim}[definition]{Claim}
\newtheorem{problem}[definition]{Problem}

\newtheorem{fact}[definition]{Fact}
\newenvironment{pf}{\noindent{\bf Proof}}

\begin{document}
\title{\bf\Large A generalized Tur\'{a}n extension of the Deza--Erd\H{o}s--Frankl Theorem}
\date{\today}
\author{Charlotte Helliar\thanks{Email: \texttt{charlotte.helliar@warwick.ac.uk}} }
\author{Xizhi Liu\thanks{Research was supported by ERC Advanced Grant 101020255. Email: \texttt{xizhi.liu.ac@gmail.com}}}
\affil{Mathematics Institute and DIMAP,
            University of Warwick,
            Coventry, CV4 7AL, UK}
\maketitle
\begin{abstract}
For an integer $r \ge 3$ and a subset $L \subset [0,r-1]$, a graph $G$ is $(K_{r}, L)$-intersecting if the number of vertices in the intersection of every pair of $K_r$ in $G$ belongs to $L$. 
We study the maximum number of $K_r$ in an $n$-vertex $(K_{r}, L)$-intersecting graphs. 
The celebrated Ruzsa--Szemer\'{e}di Theorem corresponds to the case $r=3$ and $L = \{0,1\}$. 

For general $L$ with $2 \le |L| \le r-1$, we establish the upper bound $\left(1-\frac{1}{3r}\right) \prod_{\ell \in L}\frac{n-\ell}{r- \ell}$ for large $n$, which improves the bound provided by the celebrated Deza--Erd\H{o}s--Frankl Theorem by a factor of $1-\frac{1}{3r}$. 
In the special case where $L = \{t, t+1, \ldots, r-1\}$, we derive the tight upper bound for large $n$ and establish a corresponding stability result. This is an extension of the seminal Erd{\H o}s--Ko--Rado Theorem on $t$-intersecting systems to the generalized Tur\'{a}n setting.

Our proof for the Deza--Erd\H{o}s--Frankl part involves an interesting combination of the $\Delta$-system method and Tur\'{a}n's theorem. Meanwhile, for the Erd{\H o}s--Ko--Rado part, we employ the stability method, which relies on a theorem of Frankl regarding $t$-intersecting systems.




\end{abstract}
\section{Introduction}
The seminal Tur\'{a}n Theorem~\cite{TU41} stands as a cornerstone in Extremal Graph Theory. It asserts that for $n \ge \ell \ge 2$, the maximum number of edges in an $n$-vertex $K_{\ell+1}$-free graph is achieved by the Tur\'{a}n graph $T(n,\ell)$. Here, $T(n,\ell)$ represents the balanced complete $\ell$-partite graph on $n$ vertices.
Erd\H{o}s~\cite{Erdos62} considered an extension of the Tur\'{a}n Theorem and proved that for $r \le \ell$, the maximum number of $K_{r}$ subgraphs in an $n$-vertex $K_{\ell+1}$-free graph is also attained by $T(n,\ell)$. 
In general, given a graph $Q$ and a family $\mathcal{F}$ of graphs, let $\mathrm{ex}(n,Q, \mathcal{F})$ denote the maximum number of copies of $Q$ in an $n$-vertex $\mathcal{F}$-free graph.
Alon--Shikhelman~\cite{AS16} initiated a systematic study on $\mathrm{ex}(n,Q, \mathcal{F})$ and obtained numerous results for various combinations of $(Q, \mathcal{F})$. Since then, this topic has been extensively studied in the recent decade.

We consider an interesting case of the generalized Tur\'{a}n problem motivated by theorems from Extremal Set Theory. 
Recall that an $r$-graph is a collection of $r$-subsets of some finite set $V$. 
We identify an $r$-graph with its edge set. 
Given a subset $L \subset [0,r-1]$, an $r$-graph $\mathcal{H}$ is \textbf{$L$-intersecting} if $|e\cap e'| \in L$ for every distinct pair of edges $e, e' \in \mathcal{H}$. 
If $L = \{t, t+1, \ldots, r-1\}$ for some $t \in [r]$, then we simply say $\mathcal{H}$ is \textbf{$t$-intersecting}. 
Determining the maximum size of an $L$-intersecting $r$-graph on $n$ vertices is a central topic in Extremal Set Theory and is connected to many classical results. Below, we highlight two related theorems and refer interested readers to the survey~\cite{FT16} for more insights into this area.

The seminal Erd\H{o}s--Ko--Rado Theorem established a tight upper bound for the size of $t$-intersecting $r$-graphs with a large number of vertices. 

\begin{theorem}[Erd\H{o}s--Ko--Rado~\cite{EKR61}]\label{THM:EKR61}
    Let $r > t \ge 1$ be integers. 
    There exists $n_{0}(r,t)$ such that every $t$-intersecting $r$-graph on $n\ge n_{0}(r,t)$ vertices has size at most $\binom{n-t}{r-t}$. 
\end{theorem}

The celebrated Deza--Erd\H{o}s--Frankl Theorem established a general upper bound for the size of $L$-intersecting $r$-graphs, though finding matching lower bound constructions in general is an extremely challenging problem.

\begin{theorem}[Deza--Erd\H{o}s--Frankl~\cite{DEF78}]\label{THM:DEF78}
    Suppose that $r \ge 3$ and $n \ge 2^r r^3$ are integers and $L \subset [0,r-1]$ is a subset.
    Then every $n$-vertex $L$-intersecting $r$-graph has size at most $\prod_{\ell\in L}\frac{n-\ell}{r-\ell}$. 
\end{theorem}

Given a graph $G$ and an integer $r\ge 3$, we define the associated $r$-graph $\mathcal{K}_{G}^{r}$ as follows:
\begin{align*}
    \mathcal{K}_{G}^{r}
    \coloneqq \left\{S \in \binom{V(G)}{r} \colon G[S] \cong K_{r}\right\}. 
\end{align*}
For a subset $L \subset [0,r-1]$, we say a graph $G$ is \textbf{$(K_{r}, L)$-intersecting} if the $r$-graph $\mathcal{K}_{G}^{r}$ is $L$-intersecting. 
In other words, the number of vertices in the intersection of every pair of $K_r$ in $G$ belongs to $L$.
In the special case where $L = \{t, t+1, \ldots, r-1\}$, we say $G$ is \textbf{$(K_{r}, t)$-intersecting}. 
For convenience, we use $N(K_r, G)$ to denote the number of copies of $K_{r}$ in $G$, and define  
\begin{align*}
    \Psi_{r}(n,L)
    \coloneqq \max\left\{N(K_r, G) \colon \text{$G$ is an $n$-vertex $(K_{r}, L)$-intersecting graph}\right\}. 
\end{align*}
%
There are several sporadic results on $\Psi_{r}(n,L)$ that are known. 
The celebrated Ruzsa--Szemer\'{e}di Theorem~\cite{RS78} is equivalent to the statement that $\Psi_{3}\left(n,\{0,1\}\right) = n^{2-o(1)}$, and this was extended in a recent work by Gowers--Janzer~\cite{GJ20}, who proved that $\Psi_{r}\left(n,\{0,1, \ldots, s-1\}\right) = n^{s-o(1)}$ for all $r > s \ge 1$.
An old result of Erd\H{o}s--S\'{o}s~\cite{Sos76} determines the exact value of $\Psi_{3}\left(n,\{0,2\}\right)$.
Liu--Wang~\cite{LW21} extended their result and determined the exact value of $\Psi_{r}\left(n,L\right)$ for large $n$ in cases where $L = [0,r-1]\setminus \{1\}$ and $L = [r-1]$.
They also obtained some bounds for the case where $L = [0,r-1] \setminus \{s\}$ when $2 \leq s \leq r-1$.

In this note, we aim to establish a general upper bound for $\Psi_{r}(n,L)$ and determine its exact value for the case $L = \{t, t+1, \ldots, r-1\}$, extending the Deza--Erd\H{o}s--Frankl Theorem and the Erd\H{o}s--Ko--Rado Theorem, respectively. We hope our results could serve as a starting point for a more systematic extension of theorems from Extremal Set Theory, an extremally rich theory that has been studied for decades, to generalized Tur\'{a}n problems. A more detailed discussion is included in the last section.

For integers $n \ge r > t \ge 0$, let $S(n,r,t)$ denote the graph obtained from the vertex-disjoint union of $K_{t}$ and $T(n-t,r-t)$ by adding all pairs that have a nonempty intersection with both graphs. 
It is easy to see that $S(n,r,t)$ is $(K_r, t)$-intersecting and 
\begin{align*}
    N\left(K_r, S(n,r,t)\right)
    = N\left(K_{r-t}, T(n-t,r-t)\right)
    =(1+o(1)) \left(\frac{n-t}{r-t}\right)^{r-t}. 
\end{align*}
In the following theorem, we determine the value of $\Psi(n,\{t, t+1, \ldots, r-1\})$ for large $n$. 

\begin{theorem}\label{THM:EKR-Turan}
    Let $r > t \ge 1$ be integers. 
    There exists $n_{0}(r,t)$ such that every $(K_{r},t)$-intersecting graph on $n \ge n_{0}(r,t)$ vertices satisfies 
    \begin{align*}
        N(K_{r},G) \le N\left(K_r, S(n,r,t)\right),  
    \end{align*}
    and equality holds iff $G \cong S(n,r,t)$. 
\end{theorem}

The corresponding stability version is as follows. 
\begin{theorem}\label{THM:EKR-Turan-Stability}
    Let $r > t \ge 1$ be integers. 
    There exist $\varepsilon_{0}>0$ and $n_{0}(r,t)$ such that the following holds for all $\varepsilon<\varepsilon_{0}$ and $n \ge n_{0}(r,t)$.
    If $G$ is an $n$-vertex $(K_{r},t)$-intersecting graph with $N(K_{r},G) \ge (1-\varepsilon) \left(\frac{n-t}{r-t}\right)^{r-t}$, then there exists a $t$-set $T\subset V(G)$ such that $G-T$ is $\ell$-partite after removing at most $2\varepsilon (n-t)^2$ edges. 
\end{theorem}

An extension of the Deza--Erd\H{o}s--Frankl Theorem in the generalized Tur\'{a}n setting is as follows. 

\begin{theorem}\label{THM:DEF-Turan}
    Suppose that $r \ge 3$ and $n \ge (2r)^{r+1}$ are integers and $L = \{\ell_1, \ldots, \ell_{s}\}\subset [0,r-1]$ is a subset of size $s\in [2,r-1]$.
    Then every $n$-vertex $(K_{r},t)$-intersecting graph $G$ satisfies 
    \begin{align*}
        N(K_r, G)
        \le \left(1-\frac{1}{3r}\right) \prod_{\ell\in L}\frac{n-\ell}{r-\ell}. 
    \end{align*}
\end{theorem}
\textbf{Remarks.}
\begin{itemize}
    \item It is clear that Theorem~\ref{THM:DEF-Turan} does not hold for $L = [0,r-1]$, as demonstrated by the complete graph $K_{n}$. 
    Similarly, it does not hold when $L = \{\ell\}$ for any $\ell\in [0,r-1]$, as demonstrated by the graph obtained from the vertex-disjoint union of $K_{\ell}$ and $\lfloor\frac{n-\ell}{r-\ell}\rfloor K_{r-\ell}$ by adding all pairs that have nonempty intersection with both graphs. 
    Here, $\lfloor\frac{n-\ell}{r-\ell}\rfloor K_{r-\ell}$ represents the graph consisting of $\lfloor\frac{n-\ell}{r-\ell}\rfloor$ vertex-disjoint copies of $K_{r-\ell}$.
    \item Assuming that $r > s \ge 2$ and $0 \le \ell_1 < \cdots < \ell_{s} \le r-1$, Deza--Erd\H{o}s--Frankl~\cite{DEF78} proved that for large $n$, if an $n$-vertex $\{\ell_1, \ldots, \ell_s\}$-intersecting $r$-graph has size at least $2^{s-1}r^2 n^{s-1}$, then 
    \begin{align*}
        (\ell_{2}-\ell_1) \mid (\ell_{3}-\ell_{2}) \mid \cdots \mid (\ell_{s}-\ell_{s-1}) \mid (r-\ell_{s}). 
    \end{align*}
    For $(K_{r},\{\ell_1, \ldots, \ell_s\})$-intersecting graphs, it is possible that the bound $2^{s-1}r^2 n^{s-1}$ could be improved by a factor $1-\varepsilon$ for some constant $\varepsilon>0$ using ideas from their proof and the proof of Theorem~\ref{THM:Deza-Turan-L-two}. 
    However, this seems less interesting given that they did not attempt to optimize the coefficient $2^{s-1}r^2$ in their original proof.
\end{itemize}
 
In the next section, we present some necessary definitions and preliminary results. 
In Section~\ref{SEC:Proof-EKR-Turan}, we prove Theorems~\ref{THM:EKR-Turan} and~\ref{THM:EKR-Turan-Stability}. 
In Section~\ref{SEC:DEF-Turan}, we prove Theorem~\ref{THM:DEF-Turan}. 
Some open problems are included in the last section. 
\section{Preliminaries}\label{SEC:Prelim}
In the proof of Theorems~\ref{THM:EKR-Turan} and~\ref{THM:EKR-Turan-Stability}, we will use the following result proved by Erd\H{o}s in~\cite{Erdos62}. 

\begin{theorem}[Erd\H{o}s~\cite{Erdos62}]\label{THM:Erdos-Generlaized-Turan}
    Let $n \ge \ell \ge r \ge 2$ be integers. 
    Every $n$-vertex $K_{\ell+1}$-free graph $G$ satisfies 
    \begin{align*}
        N(K_r, G) \le N\left(K_r, T(n,\ell)\right). 
    \end{align*} 
\end{theorem}

Building upon a nice theorem of F\"{u}redi~\cite{Furedi15} and an earlier result by Fisher--Ryan~\cite{FR92}, we establish the following stability result regarding Theorem~\ref{THM:Erdos-Generlaized-Turan}.
It is worth mentioning that a weaker version of this result is already well-known in Extremal Graph Theory and has appeared in various works, including~\cite{Mubayi06,MQ20,Liu21}.

\begin{theorem}\label{THM:Furedi-generalized-Turan-extension}
    Let $\ell \ge r \ge 2$ be integers. 
    Suppose that $G$ is an $n$-vertex $K_{\ell+1}$-free graph with $N(K_{r}, G) \ge \binom{\ell}{r}\left(\frac{n}{\ell}\right)^r - m$ for some integer $m \ge 0$. 
    Then $G$ can be made $\ell$-partite by removing at most $\frac{2\binom{\ell}{2}\left(\frac{n}{\ell}\right)^{2}m}{r\binom{\ell}{r}\left(\frac{n}{\ell}\right)^{r}+(r-2)m} \le \frac{2\binom{\ell}{2}m}{r\binom{\ell}{r}\left(\frac{n}{\ell}\right)^{r-2}}$ edges. 
    In particular, for every $\varepsilon>0$, every $n$-vertex $K_{r+1}$-free graph $G$ with $N(K_r, G) \ge (1-\varepsilon)\left(\frac{n}{r}\right)^{r}$ is $r$-partite after removing at most $\varepsilon(r-1)\left(\frac{n}{r}\right)^{2}$ edges. 
\end{theorem}

The following theorem may be considered as a stability version of Theorem~\ref{THM:DEF78}. 

\begin{theorem}[Deza--Erd\H{o}s--Frankl~\cite{DEF78}]\label{THM:DEF-stability}
    Let $r \ge 3$ and $n \ge 2^r r^3$ be integers. 
    Suppose that $L = \{\ell_1, \ldots, \ell_{s}\} \subset [0,r-1]$ is a subset of size $s \le r$ and $\ell_1 < \cdots < \ell_{s}$. 
    Then every $L$-intersecting $r$-graph $\mathcal{H}$ on $n$ vertices with $|\mathcal{H}| \ge 2^{s-1}r^2 n^{s-1}$ satisfies $\left|\bigcap_{e\in \mathcal{H}} e\right| \ge \ell_{1}$. 
\end{theorem}
\textbf{Remark.}
The bound $|\mathcal{H}| \ge 2^{s-1}r^2 n^{s-1}$ can be improved in some cases. 
In particular, extending the classical Hilton--Milner Theorem~\cite{HM67}, Frankl obtained a tight lower bound for the case $L = \{t, t+1, \ldots, r-1\}$ in~\cite{Frankl78}. 

In the proof of Theorem~\ref{SEC:DEF-Turan}, we will use the following old result of Deza~\cite{Deza74}. 

An $r$-graph $\mathcal{H}$ is an \textbf{$\ell$-sunflower}\footnote{Sunflowers are also called $\Delta$-systems.} for some $\ell \le r-1$ if there exists an $\ell$-set $C\subset V(\mathcal{H})$ such that $e\cap e' = C$ for all distinct edges $e.e'\in \mathcal{H}$.
The set $C$ is called the \textbf{core} of $\mathcal{H}$. 

\begin{theorem}[Deza~\cite{Deza74}]\label{THM:Deza74}
    Let $r > \ell \ge 1$ be integers. 
    Every $\{\ell\}$-intersecting $n$-vertex $r$-graph with at least $r^2-r+1$ edges is an $\ell$-sunflower. 
\end{theorem}

Given an $r$-graph $\mathcal{H}$, the \textbf{link} of $v\in V(\mathcal{H})$ is defined as 
\begin{align*}
    L_{\mathcal{H}}(v) 
    \coloneqq \left\{S \in \binom{V(\mathcal{H})\setminus\{v\}}{r-1} \colon S\cup \{v\} \in \mathcal{H}\right\}. 
\end{align*}
The \textbf{degree} of $v$ is $d_{\mathcal{H}}(v) \coloneqq |L_{\mathcal{H}}(v)|$. 
\section{Proof of Theorem~\ref{THM:Furedi-generalized-Turan-extension}}\label{SEC:Furedi-generalized-Turan-extension}
We establish Theorem~\ref{THM:Furedi-generalized-Turan-extension} in this section, relying on a straightforward combination of the following results.

\begin{fact}\label{FACT:Inequality-a}
    Suppose that $r \ge 2$ and $-1 \le x < \frac{1}{r-1}$. 
    Then 
    \begin{align*}
        (1+x)^{r} \le 1 + \frac{rx}{1-(r-1)x}. 
    \end{align*}
\end{fact}

\begin{theorem}[Fisher--Ryan~\cite{FR92}]\label{THM:Fisher-Ryan}    
    Let $\ell \ge r \ge 2$ be integers. 
    Suppose that $G$ is an $n$-vertex $K_{\ell+1}$-free graph. 
    Then 
    \begin{align*}
        N(K_r, G)
        \le \binom{\ell}{r}\left({|G|}/{\binom{\ell}{2}}\right)^{\frac{r}{2}}. 
    \end{align*}
\end{theorem}

\begin{theorem}[F\"{u}redi~\cite{Furedi15}]\label{THM:Furedi}
    Suppose that $\ell \ge 2$ and $G$ is an $n$-vertex $K_{\ell+1}$-free graph with $|G| \ge |T(n,\ell)| - m$.
    Then $G$ can be made $\ell$-partite by removing at most $m$ edges. 
\end{theorem}

\begin{proof}[Proof of Theorem~\ref{THM:Furedi-generalized-Turan-extension}]
    Let $x \coloneqq \frac{2m}{r\binom{\ell}{r}\left(\frac{n}{\ell}\right)^{r}+(r-2)m}$. 
    Observe that $\binom{\ell}{r}\left(\frac{n}{\ell}\right)^r\frac{xr/2}{1-(r/2-1)x} = m$ and $x\le \frac{2}{r} < \frac{1}{r/2-1}$. 
    First, we prove that $|G| \ge (1-x)\binom{\ell}{2}\left(\frac{n}{\ell}\right)^{2}$. 
    Indeed, suppose to the contrary that $|G| <  (1-x)\binom{\ell}{2}\left(\frac{n}{\ell}\right)^{2}$. 
    Then it follows from Theorem~\ref{THM:Fisher-Ryan} and Lemma~\ref{FACT:Inequality-a} that 
    \begin{align*}
        N(K_r, G)
         \le \binom{\ell}{r}\left({|G|}/{\binom{\ell}{2}}\right)^{\frac{r}{2}} 
        & < \binom{\ell}{r}\left(\frac{n}{\ell}\right)^{r}\left(1- x\right)^{\frac{r}{2}} \\
        & \le \binom{\ell}{r}\left(\frac{n}{\ell}\right)^{r}\left(1- \frac{xr/2}{1-(r/2-1)x}\right) \\
        & = \binom{\ell}{r}\left(\frac{n}{\ell}\right)^r - m, 
    \end{align*}
    a contradiction. 
    Therefore, we have $|G| \ge (1-x)\binom{\ell}{2}\left(\frac{n}{\ell}\right)^{2} \ge |T(n,\ell)| - x \binom{\ell}{2}\left(\frac{n}{\ell}\right)^{2}$. 
    Hence, it follows from Theorem~\ref{THM:Furedi} that $G$ can be made $\ell$-partite by removing at most $x\binom{\ell}{2}\left(\frac{n}{\ell}\right)^{2} = \frac{2\binom{\ell}{2}\left(\frac{n}{\ell}\right)^{2}m}{r\binom{\ell}{r}\left(\frac{n}{\ell}\right)^{r}+(r-2)m}$ edges. 
\end{proof}

\section{Proofs of Theorems~\ref{THM:EKR-Turan} and~\ref{THM:EKR-Turan-Stability}}\label{SEC:Proof-EKR-Turan}
We prove Theorems~\ref{THM:EKR-Turan} and~\ref{THM:EKR-Turan-Stability} in this section.

\begin{proof}[Proofs of Theorems~\ref{THM:EKR-Turan} and~\ref{THM:EKR-Turan-Stability}]
Fix integers $r > t \ge 1$, let $\varepsilon>0$ be sufficiently small and $n$ be a sufficiently large integer. 
Let $G$ be an $n$-vertex $(K_r, t)$-intersecting graph with $N(K_r, G) \ge (1-\varepsilon)\left(\frac{n-t}{r-t}\right)^{r-t}$. 
%
Note that the associated $r$-graph $\mathcal{K}_{G}^{r}$ is $t$-intersecting and satisfies 
\begin{align}\label{equ:KG-lower-bound}
    |\mathcal{K}_{G}^{r}|
     = N(K_r, G)
     \ge (1-\varepsilon)\left(\frac{n-t}{r-t}\right)^{r-t}. 
\end{align}
Since $n$ is large, we have $|\mathcal{K}_{G}^{r}| \ge 2^{r-t}r^2 n^{r-t-1}$. 
It follows from Theorem~\ref{THM:DEF-stability} that 
there exists a $t$-set $T \subset V(G)$ such that $T$ is contained in all edges in $\mathcal{K}_{G}^{r}$. In other words, every copy of $K_{r}$ in $G$ contains $T$. 
Let $W\coloneqq V(G)\setminus T$ and $N \coloneqq \bigcap_{v\in T}N_{G}(v)$. Note that $|N| \le n-t$ and 
\begin{align}\label{equ:KG-lower-bound-2}
     N(K_r, G)
    = N\left(K_{r-t}, G[N]\right). 
\end{align}
In the sequel, we will prove that $G[N]$ is $K_{r-t+1}$-free. 

Suppose to the contrary that there exists a $(r-t+1)$-set $S\subset N$ such that $G[S] \cong K_{r-t+1}$. 
Then every copy of $K_{r-t}$ in $G[S]$ must contain at least one vertex from $S$. 
Indeed, if there exists a $(r-t)$-set $S' \subset N\setminus S$ such that $G[S'] \cong K_{r-t}$. 
Then take an arbitrary $T' \subset T$ of size $t-1$, and observe that $G[T' \cup S] \cong G[T\cup S'] \cong K_{r}$ while $(T' \cup S) \cap (T\cup S')$ has size only $t-1$, a contradiction. 
Therefore, every copy of $K_{r-t}$ in $G[S]$ must contain at least one vertex from $S$. 
Consequently, we have 
\begin{align*}
    N\left(K_{r-t}, G[N]\right)
    \le |S|\cdot |N|^{r-t-1}
    \le (r-t+1)n^{r-t-1}, 
\end{align*}
contradicting~\eqref{equ:KG-lower-bound}. 
Therefore, $G[N]$ is $K_{r-t+1}$-free. 

It follows from Theorem~\ref{THM:Erdos-Generlaized-Turan} and~\eqref{equ:KG-lower-bound-2} that 
\begin{align*}
    N(K_r, G)
    = N(K_{r-t}, G[N])
    \le N\left(K_{r-t}, T(n-t, r-t)\right), 
\end{align*}
completing the proof of Theorem~\ref{THM:EKR-Turan}. 

To prove Theorem~\ref{THM:EKR-Turan-Stability}, observe from~\eqref{equ:KG-lower-bound} and~\eqref{equ:KG-lower-bound-2} that $N(K_{r-t}, G[N]) \ge (1-\varepsilon)\left(\frac{n-t}{r-t}\right)^{r-t}$. 
It follows from Theorem~\ref{THM:Erdos-Generlaized-Turan} that $|N| \ge (1-\varepsilon)(n-t)$, since otherwise we would have 
\begin{align*}
    \left(\frac{(1-\varepsilon)(n-t)}{r-t}\right)^{r-t}
    \ge N(K_{r-t}, G[N]) 
    \ge (1-\varepsilon)\left(\frac{n-t}{r-t}\right)^{r-t},
\end{align*}
a contradiction. 
Therefore, $|G[W]\setminus G[N]| \le |W\setminus N| n \le \varepsilon (n-t)^2$. 
On the other hand, since $G[N]$ is $K_{r-t+1}$-free, it follows from Theorem~\ref{THM:Furedi-generalized-Turan-extension} that $G[N]$ is $(r-t)$-partite after removing $\varepsilon (r-t-1)\left(\frac{n-t}{r-t}\right)^{2}$ edges. 
In total, we can remove at most $\varepsilon (n-t)^2 + \varepsilon (r-t-1)\left(\frac{n-t}{r-t}\right)^{2} \le 2\varepsilon (n-t)^2$ edges to make $G[W]$ become $\ell$-partite. 
\end{proof}

\section{Proof of Theorem~\ref{THM:DEF-Turan}}\label{SEC:DEF-Turan}
We establish Theorem~\ref{THM:DEF-Turan} in this section, employing induction on both $r$ and the size of the set $L$.
The base case is presented in the following theorem, the proof of which is an interesting combination of the $\Delta$-system method and Tur\'{a}n's theorem.
\begin{theorem}\label{THM:Deza-Turan-L-two}
    Let $r \ge 3$, $0 \le \ell_1 < \ell_2 \le r-1$, and $n \ge 5r^4$ be integers. 
    Suppose that $G$ is an $n$-vertex $(K_r, \{\ell_1, \ell_2\})$-intersecting graph. 
    Then 
    \begin{align*}
        N(K_r, G)
        \le \left(1-\frac{\ell_2-\ell_1}{2(r-\ell_1)}\right) \frac{(n-\ell_1)^2}{(r-\ell_1)(r-\ell_2)}
        \le \left(1-\frac{1}{3r}\right) \frac{(n-\ell_1)(n-\ell_2)}{(r-\ell_1)(r-\ell_2)}. 
    \end{align*}
\end{theorem}
\begin{proof}[Proof of Theorem~\ref{THM:Deza-Turan-L-two}]
     Fix integers $r \ge 3$ and $0 \le \ell_1 < \ell_2 \le r-1$. 
     Suppose to the contrary that there exists a $(K_r, \{\ell_1, \ell_2\})$-intersecting graph $G$ on $n \ge 5r^4$ vertices with $N(K_r, G) \ge \left(1-\frac{\ell_2-\ell_1}{2(r-\ell_1)}\right) \frac{(n-\ell_1)^2}{(r-\ell_1)(r-\ell_2)}$. 
     Note that the associated $r$-graph $\mathcal{K}_{G}^{r}$ is $\{\ell_1, \ell_2\}$-intersecting with at least $\left(1-\frac{\ell_2-\ell_1}{2(r-\ell_1)}\right) \frac{(n-\ell_1)^2}{(r-\ell_1)(r-\ell_2)}$ edges. 

     \textbf{Case 1:} $\ell_1 = 0$. 

     Let $\ell \coloneqq \ell_2$, and let
     \begin{align*}
        V \coloneqq V(G), \quad 
         U 
         \coloneqq \left\{u\in V \colon d_{\mathcal{K}_{G}^{r}}(u) \ge r^2\right\}, \quad
         \overline{U} \coloneqq V\setminus U. 
     \end{align*}
     For every $C\subset V$ of size $\ell$, fix a maximum sunflower $\mathcal{S}_{C} \subset \mathcal{H}$ with $C$ as its core. 
     Let 
     \begin{align*}
         \mathcal{C}
         \coloneqq \left\{C\in \binom{V}{\ell} \colon |\mathcal{S}_{C}| \ge r^2\right\}. 
     \end{align*}
    \begin{claim}\label{CLAIM:U-is-union-C}
        We have $U = \bigcup_{C\in \mathcal{C}}C$. 
        In particular, $|U| = \ell |\mathcal{C}|$. 
    \end{claim}
    \begin{proof}
        The inclusion $\bigcup_{C\in \mathcal{C}}C \subset U$ follows easily from the definitions, so it suffices to show the other direction. 
        Fix $u\in U$. 
        Observe that the link $L_{\mathcal{H}}(u)$ is an $\{\ell-1\}$-intersecting  $(r-1)$-graph. 
        So it follows from the definition of $U$ and Theorem~\ref{THM:Deza74} that $L_{\mathcal{H}}(u)$ is an $(\ell-1)$-sunflower of size at least $r^2$. Therefore, $u$ is contained in the core of an $\ell$-sunflower of size at least $r^2$. This implies that $U \subset \bigcup_{C\in \mathcal{C}}C$. 
    \end{proof}

    \begin{claim}\label{CLAIM:cores-are-disjoint}
        We have $C\cap C' = \emptyset$ for all distinct $C, C' \in \mathcal{C}$. 
    \end{claim}
    \begin{proof}[Proof of Claim~\ref{CLAIM:cores-are-disjoint}]
        Suppose to the contrary that there exist distinct $C, C' \in \mathcal{C}$ with $C\cap C' \neq \emptyset$ (which can only occur if $\ell>1$). 
        Since $\mathcal{S}_{C}$ is a sunflower of size at least $r^2$, there exists $E\in \mathcal{S}_{C}$ such that $E\setminus C$ is disjoint from $C'$, that is, $E\cap C' = C\cap C'$. 
        Similarly, since $\mathcal{S}_{C'}$ is a sunflower of size at least $r^2$, there exists $E'\in \mathcal{S}_{C'}$ such that $E'\cap E = C\cap C'$. 
        However, this implies that $1\le |E\cap E'| \le \ell-1$, a contradiction. 
    \end{proof}

    \begin{claim}\label{CLAIM:cores-edges}
        For every $E\in \mathcal{K}_{G}^{r}$ and every $C\in \mathcal{C}$, we have either $E \in \mathcal{S}_{C}$ or $E \cap C = \emptyset$. 
    \end{claim}
    \begin{proof}
        It suffices to show that either $E \cap C = \emptyset$ or $C \subset E$, as the latter case implies that $E \in \mathcal{S}_{C}$ (since $\mathcal{K}_{G}^{r}$ is $\{0,\ell\}$-intersecting). 
        Suppose to the contrary that there exist $E\in \mathcal{K}_{G}^{r}$ and $C\in \mathcal{C}$ such that $1\le |C\cap E| \le \ell-1$ (which can only occur if $\ell>1$). 
        Similar to the proof of Claim~\ref{CLAIM:cores-are-disjoint}, since $\mathcal{S}_{C}$ is a sunflower of size at least $r^2$, there exists $E'\in \mathcal{S}_{C}$ such that $E'\cap E = C'\cap E$. 
        This implies that $1\le |E\cap E'| \le \ell-1$, a contradiction.
    \end{proof}

    Define an auxiliary graph $H$ with vertex set $\mathcal{C}$. Two distinct members $C, C' \in \mathcal{C}$ are adjacent in $H$ if there exists $E \in \mathcal{S}_{C}$ such that $C' \subset E$, or there exists $E' \in \mathcal{S}_{C'}$ such that $C \subset E'$.

    \begin{claim}\label{CLAIM:aux-graph-clique-number}
        The graph $H$ is $K_{\lceil 2r/\ell \rceil}$-free.
        In particular, 
        \begin{align*}
            |H|
            \le \left(1 - \frac{1}{\lceil 2r/\ell \rceil - 1}\right) \frac{|\mathcal{C}|^2}{2}
            \le \left(1 - \frac{\ell}{2r}\right) \frac{|\mathcal{C}|^2} {2}. 
        \end{align*}
    \end{claim}
    \begin{proof}
        Suppose to the contrary that $\{C_{1}, \ldots, C_{\lceil 2r/\ell \rceil}\} \subset \mathcal{C}$ induces a copy of $K_{\lceil 2r/\ell \rceil}$ in $H$. 
        Let $S \coloneqq C_{1}\cup \cdots \cup C_{\lceil 2r/\ell \rceil}$. 
        It follows from the definitions of $\mathcal{K}_{G}^{r}$ and $H$ that the induced subgraph $G[S]$ is complete and contains $\ell \cdot \lceil 2r/\ell \rceil \ge 2r$ vertices. 
        It is easy to see that $G[S]$ contains two copies of $K_{r}$ whose intersection is not $0$ or $\ell$ (here, we used the assumption that $r \ge 3$), a contradiction. 
    \end{proof}

    \begin{claim}\label{CLAIM:degree-u-sunflower}
        For every $C\in \mathcal{C}$ and $u\in C$ we have $d_{\mathcal{K}_{G}^{r}}(u) \le \frac{\ell \cdot d_{H}(C) + |\overline{U}|}{r-\ell}$. 
    \end{claim}
    \begin{proof}
        Fix $C\in \mathcal{C}$ and $u\in C$. 
        It follows from Claim~\ref{CLAIM:cores-edges} that $d_{\mathcal{K}_{G}^{r}}(u) = |\mathcal{S}_{C}|$. 
        The definition of $H$ and Claim~\ref{CLAIM:cores-edges} imply that $E\setminus C \subset \bigcup_{C'\in N_{H}(C)}C' \cup \overline{U}$ for every $E\in \mathcal{S}_{C}$. 
        Since $\mathcal{S}_{C}$ is a sunflower, we have $|\mathcal{S}_{C}| \le \frac{|\bigcup_{C'\in N_{H}(C)}C' \cup \overline{U}|}{r-\ell} = \frac{\ell \cdot d_{H}(C) + |\overline{U}|}{r-\ell}$. 
    \end{proof}

    By Claims~\ref{CLAIM:U-is-union-C},~\ref{CLAIM:aux-graph-clique-number}, and~\ref{CLAIM:degree-u-sunflower}, we have 
    \begin{align*}
        \sum_{u\in U}d_{\mathcal{K}_{G}^{r}}(u)
        = \sum_{C\in \mathcal{C}}\sum_{u\in C}d_{\mathcal{K}_{G}^{r}}(u) 
        & \le \sum_{C\in \mathcal{C}} \ell\cdot \frac{\ell \cdot d_{H}(C) + |\overline{U}|}{r-\ell} \\
        & = \frac{1}{r-\ell}\left(2|H| \cdot \ell^2 + \ell |\mathcal{C}| |\overline{U}|\right) \\
        & \le \frac{1}{r-\ell}\left(2 \left(1 - \frac{\ell}{2r}\right) \frac{|\mathcal{C}|^2}{2} \cdot \ell^2  + \ell |\mathcal{C}| |\overline{U}|\right) \\
        & = \frac{1}{r-\ell}\left(\left(1 - \frac{\ell}{2r}\right) |U|^2  + |U| |\overline{U}|\right). 
    \end{align*}
    Therefore, 
    \begin{align*}
        r\cdot |\mathcal{K}_{G}^{r}|
        = \sum_{u\in V}d_{\mathcal{K}_{G}^{r}}(u)
        & = \sum_{u\in U}d_{\mathcal{K}_{G}^{r}}(u) + \sum_{u\in \overline{U}}d_{\mathcal{K}_{G}^{r}}(u) \\
        & \le \frac{1}{r-\ell}\left(\left(1 - \frac{\ell}{2r}\right) |U|^2  + |U| |\overline{U}|\right)  + r^2 |\overline{U}|  \\
        & \le \frac{1}{r-\ell}\left(\left(1 - \frac{\ell}{2r}\right) |U|^2  + |U| |\overline{U}| + r^3 |\overline{U}|\right) 
         \le \left(1 - \frac{\ell}{2r}\right) \frac{n^2}{r-\ell},  
    \end{align*}
    where the last inequality follows from a simple analysis of the quadratic form $\left(1 - \frac{\ell}{2r}\right) |U|^2  + |U| |\overline{U}| + r^3 |\overline{U}|$ (by considering $|U|$ as the variable) and the assumption that $n \ge 5r^4$. 
    Therefore, 
    \begin{align}\label{equ:Deza-two-a}
        N(K_r, G) 
        = |\mathcal{K}_{G}^{r}| 
        \le \left(1 - \frac{\ell}{2r}\right) \frac{n^2}{r(r-\ell)}
        \le  \left(1 - \frac{1}{3r}\right) \frac{n(n-\ell)}{r(r-\ell)}
    \end{align}
    where the last inequality follows from the assumption that $n \ge 5r^4$. 

    \medskip 

    \textbf{Case 2:} $\ell_1 \ge 1$.

    Let $\ell \coloneqq \ell_2 - \ell_1 \ge 1$. 
    Since 
    \begin{align*}
        |\mathcal{K}_{G}^{r}|
        \ge \left(1-\frac{\ell_2-\ell_1}{2(r-\ell_1)}\right) \frac{(n-\ell_1)^2}{(r-\ell_1)(r-\ell_2)}
        \ge \frac{1}{2} \frac{(n-r)^2}{r^2}
        > 2r^2 n, 
    \end{align*}
    it follows from Theorem~\ref{THM:DEF-stability} that there exists an $\ell_1$-set $T \subset V$ such that $T$ is contained in every edge in $\mathcal{K}_{G}^{r}$. In other words, every copy of $K_r$ in $G$ contains $T$. 
    Let $N \coloneqq \bigcap_{v \in T}N_{G}(v) \subset V\setminus T$. 
    It is easy to see that $G[N]$ is $(K_{r-\ell_1}, \{0,\ell\})$-intersecting. 
    Therefore, by~\eqref{equ:Deza-two-a}, we have 
    \begin{align*}
        N(K_r, G)
        = N(K_{r-t}, G[N])
        & \le \left(1 - \frac{\ell}{2(r-\ell_1)}\right) \frac{(n-\ell_1)^2}{(r-\ell_1)(r-\ell_1-\ell)} \\
        & = \left(1 - \frac{\ell_2-\ell_1}{2(r-\ell_1)}\right) \frac{(n-\ell_1)^2}{(r-\ell_1)(r-\ell_2)} \\
        & \le \left(1 - \frac{1}{3r}\right) \frac{(n-\ell_1)(n-\ell_2)}{(r-\ell_1)(r-\ell_2)}, 
    \end{align*}
    completing the proof of Theorem~\ref{THM:Deza-Turan-L-two}. 
\end{proof}

Now we are ready to prove Theorem~\ref{THM:DEF-Turan}. 
\begin{proof}[Proof of Theorem~\ref{THM:DEF-Turan}]
    We prove by induction on $r$ and $|L|$. 
    The base case $|L| = 2$ and $r \ge 3$ was proved in Theorem~\ref{THM:Deza-Turan-L-two}, so we may assume that $L = \{\ell_1, \cdots, \ell_{s}\} \subset [0, r-1]$ satisfies $s\ge 3$ and $\ell_1 < \cdots < \ell_{s}$. 
    Let $G$ be a $(K_r, L)$-intersecting graph on $n \ge (2r)^{r+1}$ vertices. 
    Let $V \coloneqq V(G)$. 

    \textbf{Case 1:} $0 \in L$. 
    
    Let $L' \coloneqq \{\ell_2-1, \ldots, \ell_{s}-1\}$. 
    Fix $u \in V$ and let $N_{u} \coloneqq N_{G}(u)$. 
    Notice that the induced subgraph $G[N_{u}]$ is $(K_{r-1}, L')$-intersecting. 
    It follows from the inductive hypothesis that 
    \begin{align*}
        d_{\mathcal{K}_{G}^{r}}(u) 
        = N(K_{r-1}, G[N_{u}])
         & \le \left(1-\frac{1}{3(r-1)}\right) \prod_{i=2}^{s}\frac{(n-1)-(\ell_{i}-1)}{(r-1) - (\ell_{i}-1)}  \\
         & \le \left(1-\frac{1}{3r}\right) \prod_{i=2}^{s}\frac{n-\ell_{i}}{r-\ell_i}. 
    \end{align*}
    Consequently, 
    \begin{align*}
        N(K_r, G)
        = |\mathcal{K}_{G}^{r}|
        = \frac{1}{r} \sum_{u\in V}d_{\mathcal{K}_{G}^{r}}(u) 
        \le \frac{n}{r} \cdot \left(1-\frac{1}{3r}\right) \prod_{i=2}^{s}\frac{n-\ell_{i}}{r-\ell_i}
        = \left(1-\frac{1}{3r}\right) \prod_{i=1}^{s}\frac{n-\ell_{i}}{r-\ell_i}. 
    \end{align*}

    \medskip 

    \textbf{Case 2:} $0 \not\in L$. 

    We may assume that $|\mathcal{K}_{G}^{r}| \ge 2^{s-1}r^2 n^{s-1}$ since otherwise it follows from $n \ge (2r)^{r+1}$ that 
    \begin{align*}
        N(K_r, G)
        =|\mathcal{K}_{G}^{r}| 
        < 2^{s-1}r^2 n^{s-1} 
        \le \left(1-\frac{1}{3r}\right)\prod_{i=1}^{s}\frac{n-\ell_{i}}{r-\ell_i}.
    \end{align*}
    By Theorem~\ref{THM:DEF-stability}, there exists an $\ell_1$-set $T \subset V$ such that $T \subset E$ for all $E \in \mathcal{K}_{G}^{r}$. 
    Let $N \coloneqq \bigcap_{v\in T}N_{G}(v) \subset V\setminus T$. 
    Let $L' \coloneqq \{0, \ell_{2}-\ell_{1}, \ldots, \ell_{s}-\ell_{1}\}$. 
    It is easy to see that the induced subgraph $G[N]$ is $(K_{r-\ell_1}, L')$-intersecting. 
    It follows from the inductive hypothesis that 
    \begin{align*}
        N(K_r, G)
        = N(K_{r-\ell_1}, G[N])
        & \le \left(1-\frac{1}{3(r-\ell_1)}\right) \prod_{i=1}^{s}\frac{(n-\ell_1)-(\ell_{i}-\ell_1)}{(r-\ell_1)-(\ell_i-\ell_1)} \\
        & \le \left(1-\frac{1}{3r}\right) \prod_{i=1}^{s}\frac{n-\ell_{i}}{r-\ell_i}, 
    \end{align*}
    completing the proof of Theorem~\ref{THM:DEF-Turan}. 
\end{proof}

\section{Concluding remarks}
There are several natural questions left open regarding Theorems~\ref{THM:EKR-Turan} and~\ref{THM:DEF-Turan}.
For example, determining the optimal value of $n_0(r,t)$ in Theorem~\ref{THM:EKR-Turan} and improving the factor $1-\frac{1}{3r}$ in Theorem~\ref{THM:DEF-Turan} are interesting directions for further investigation.
It seems that our proof for Theorem~\ref{THM:EKR-Turan} yields an exponential bound (in terms of $r$) for $n_0(r,t)$.

Recall that $\Psi_{r}(n,L)$ is the maximum number of $K_r$ in an $n$-vertex $(K_{r}, L)$-intersecting graph. 
Let $\Phi_{r}(n,L)$ denote the maximum size of an $n$-vertex $L$-intersecting $r$-graph. 
It is clear that $\Psi_{r}(n,L) \le \Phi_{r}(n,L)$ for all $n \ge r \ge 2$ and $L \subset [0, r-1]$. 
The following general question seems interesting. 

\begin{problem}\label{PROB:L-intersecting-vs-generalized-Turan-lower-bound}
    Characterize the family of sets $L\subset [0,r-1]$ such that 
    \begin{align*}
        \lim_{n\to \infty} \frac{\Psi_{r}(n,L)}{\Phi_{r}(n,L)} = 0. 
    \end{align*}
\end{problem}
\textbf{Remark.}
When $L = [0,s-1]$ for some $s\in [2, r-1]$, 
the celebrated work of R\"{o}dl~\cite{Rod85} show that $\Phi(n,r,L) = (1+o(1))\binom{n}{r}/\binom{r}{s} = \Theta(n^{s})$, while a simple aurgument using the Graph Removal Lemma (see~\cite{GJ20}) shows that $\Psi(n,r,L) = o(n^{s})$. 

Another intriguing question, which seeks an analogue of the celebrated Frankl--Wilson Theorem~\cite{FW81} in the generalized Tur\'{a}n setting, is as follows. 

\begin{problem}\label{PROB:FW-graph}
    Let $n > r \ge s \ge 1$ be integers and  $p$ be a prime number. 
    Let $L \subset [0,p-1]$ be a set of $s$ integers. 
    What is the maximum value of $N(K_r, G)$ if $G$ is an $n$-vertex graph satisfying 
    \begin{enumerate}[label=(\roman*)]
        \item $k\not\in L \Mod{p}$, and 
        \item $|e\cap e'| \in L \Mod{p}$ for all distinct edges $e, e' \in \mathcal{K}_{G}^r$. 
    \end{enumerate}
\end{problem}

There are numerous classical theorems from Extremal Set Theory that one could consider extending to the generalized Tur\'{a}n setting. 
Here we list a few instances and refer the reader to surveys~\cite{FT16,MV16} for more potential extensions.

\begin{problem}\label{PROB:general-extension}
    Given an $n$-vertex graph $G$, what is the maximum value of $N(K_r, G)$ if
    \begin{enumerate}[label=(\roman*)]
        \item\label{PROB:general-extension-1} $\mathcal{K}_{G}^{r}$ contains at most $t$ pairwise vertex-disjoint edges for $2 \le t \le n/r$? 
        \item\label{PROB:general-extension-2} $\mathcal{K}_{G}^{r}$ does not contain a certain $r$-uniform tree $T$?
        \item\label{PROB:general-extension-3} $\mathcal{K}_{G}^{r}$ does not contain the cycle $C_{\ell}^{t}$? Here $C_{\ell}^{t}$ denote the $r$-uniform length-$\ell$ cycle on $[(r-t)\ell]$ with edge set 
        \begin{align*}
            \left\{\{j(r-t)+1, \cdots, (j+1)r-jt\} \Mod{(r-t)\ell} \colon 0 \le j \le \ell-1\right\}. 
        \end{align*}
    \end{enumerate}
\end{problem}
Problem~\ref{PROB:general-extension}~\ref{PROB:general-extension-1} corresponds to the famous Erd\H{o}s Matching Problem~\cite{Erd65} and was addressed in~\cite{ZCGZ23} for fixed $t$ and large $n$. 
Problem~\ref{PROB:general-extension}~\ref{PROB:general-extension-2} corresponds to the Tree problem in $r$-graphs (see e.g.~\cite{BK14,F14tree,FJS14path,FJ15,KMV17b,FJKMV19a}) and was studied in~\cite{LGHSTVZ22,LS23a,LS23b} for some special cases. 
Problem~\ref{PROB:general-extension}~\ref{PROB:general-extension-3} corresponds to Cycle problems in $r$-graphs (see e.g.~\cite{KMV15a,FJ15cycle}) and was studied in~\cite{LGHSTVZ22,LS23a,LS23b} for the case $t=1$. 
\bibliographystyle{alpha}
\newcommand{\etalchar}[1]{$^{#1}$}

\end{document}